\documentclass[11pt]{article}
\usepackage{color}
\usepackage{enumerate}
\usepackage{amsthm,amsmath,amssymb}
\usepackage{vmargin}
\usepackage{hyperref}
\usepackage{multirow}

\usepackage[T1]{fontenc} 
\usepackage[utf8]{inputenc}

\usepackage{tikz}
\usetikzlibrary{arrows, automata, shapes, positioning}
\usetikzlibrary{arrows.meta}
\usetikzlibrary{decorations.pathmorphing}
\usetikzlibrary{calc}

\setmarginsrb{3cm}{2cm}{3cm}{2cm}{0cm}{2cm}{0cm}{1cm}

\theoremstyle{plain}
\newtheorem{theorem}{Theorem}
\newtheorem{lemma}[theorem]{Lemma}

\newtheorem{proposition}[theorem]{Proposition}

\theoremstyle{definition}
\newtheorem{definition}[theorem]{Definition}
\newtheorem{example}[theorem]{Example}
\newtheorem{remark}[theorem]{Remark} 

\newcommand{\N}{\mathbb{N}}
\newcommand{\K}{\mathbb{K}}
\newcommand{\Z}{\mathbb{Z}}
\newcommand{\rep}{\mathrm{rep}}
\newcommand{\val}{\mathrm{val}}
\newcommand{\ba}{\mathbf{a}}
\newcommand{\bb}{\mathbf{b}}
\newcommand{\bu}{\mathbf{u}}
\newcommand{\bv}{\mathbf{v}}
\newcommand{\bw}{\mathbf{w}}
\newcommand{\bn}{\mathbf{n}}
\newcommand{\bz}{\mathbf{0}}
\newcommand{\bU}{\mathbf{U}}
\newcommand{\bA}{\mathbf{A}}
\newcommand{\bB}{\mathbf{B}}
\newcommand{\couple}[2]{\left(\begin{smallmatrix} #1 \\ #2 \end{smallmatrix}\right)}
\newcommand{\duple}[2]{\left(\begin{smallmatrix} #1 \\\vdots\\ #2 \end{smallmatrix}\right)}
\newcommand{\dcart}[2]{#1\times \cdots\times #2}
\newcommand\undermat[2]{%
  \makebox[0pt][l]{$\smash{\underbrace{\phantom{%
    \begin{matrix}#2\end{matrix}}}_{\text{$#1$}}}$}#2}

\def\restriction#1#2{\mathchoice
              {\setbox1\hbox{${\displaystyle #1}_{\scriptstyle #2}$}
              \restrictionaux{#1}{#2}}
              {\setbox1\hbox{${\textstyle #1}_{\scriptstyle #2}$}
              \restrictionaux{#1}{#2}}
              {\setbox1\hbox{${\scriptstyle #1}_{\scriptscriptstyle #2}$}
              \restrictionaux{#1}{#2}}
              {\setbox1\hbox{${\scriptscriptstyle #1}_{\scriptscriptstyle #2}$}
              \restrictionaux{#1}{#2}}}
\def\restrictionaux#1#2{{#1\,\smash{\vrule height .8\ht1 depth .85\dp1}}_{\,#2}}

\author{\'Emilie Charlier, Célia Cisternino and Manon Stipulanti\\
Department of Mathematics\\
University of Li\`ege\\
All\'ee de la D\'ecouverte 12\\
4000 Li\`ege, Belgium\\
{\tt \{echarlier,ccisternino,m.stipulanti\}@uliege.be}}

\title{Robustness of Pisot-regular sequences} 
\date{\today} 

\begin{document}
\thispagestyle{empty}
\maketitle

\begin{abstract} 
We consider numeration systems based on a $d$-tuple $\bU=(U_1,\ldots,U_d)$ of sequences of integers and we define $(\bU,\K)$-regular sequences through $\K$-recognizable formal series, where $\K$ is any semiring. We show that, for any $d$-tuple $\bU$ of Pisot numeration systems and any semiring $\K$, this definition does not depend on the greediness of the $\bU$-representations of integers. The proof is constructive and is based on the fact that the normalization is realizable by a $2d$-tape finite automaton. In particular, we use an ad hoc operation mixing a $2d$-tape automaton and a $\K$-automaton in order to obtain a new $\K$-automaton.
\end{abstract}

\bigskip
\hrule
\bigskip



\bigskip
\hrule
\bigskip

\section{Introduction}
A $b$-automatic sequence is a sequence $f\colon \N\to \Delta$, where $\Delta$ is a finite alphabet, that is generated by a deterministic finite automaton with output (DFAO) as follows: the $n$th term $f(n)$ of the sequence is output by the DFAO when the input is the $b$-expansion of $n$. This notion originally appeared in the work of Büchi~\cite{Buchi1960}, but Cobham was the first to systematically study $b$-automatic sequences, or equivalently, $b$-recognizable sets of integers~\cite{Cobham1969}. A large survey was undertaken by Allouche and Shallit~\cite{AlloucheShallit2003}. The most famous example is the Thue--Morse sequence, which is $2$-automatic. 

With the aim of generalizing $b$-automatic sequences to sequences having infinitely many values, Allouche and Shallit introduced the notion of $(b,\K)$-regular sequences over a N\oe therian ring $\K$~\cite{AlloucheShallit1992}. In this paper, we choose to work with the following definition, which originally was a characterization of $(b,\K)$-regular sequences in \cite{AlloucheShallit1992}. A sequence $f\colon \N\to \K$ with values in a semiring $\K$ is said to be $(b,\K)$-regular if there exist a positive integer $r$ and matrices $M_0,\ldots,M_{b-1}\in\K^{r\times r}$, $\lambda \in \K^{1 \times r}$ and $\gamma \in \K^{r \times 1}$ such that, for all words $a_\ell\cdots a_0$ over the alphabet $\{0,\ldots,b-1\}$, the matrix product $\lambda M_{a_\ell}\cdots M_{a_0} \gamma$ equals $f(n)$ if $a_\ell\cdots a_0$ is the $b$-expansion of $n$ and equals $0$ otherwise. The latter condition translates in algebraic terms by saying that the formal series (with noncommutative variables) $\sum_{w\in \rep_b(\N)}f(\val_b(w))w$ is $\K$-recognizable. In the previous series, the words $w$ run over $\rep_b(\N)$, which is the language of the $b$-expansions of non-negative integers, that is, the language made of the words over the alphabet $\{0,\ldots,b-1\}$ not starting with the letter $0$. Note that we do not impose any condition on the semiring $\K$ in our definition. Regular sequences with no condition on the underlying semiring were studied by Berstel and Reutenauer~\cite{BerstelReutenauer2011} and used in~\cite{Charlier2018,CharlierRampersadShallit2012}. In particular, it can be proven that a sequence with values in a semiring $\K$ is $b$-automatic if and only if it is $(b,\K)$-regular and takes only finitely many values. 

The notion of $(b,\K)$-regular sequences can be extended to more general numeration systems in a natural way. Such generalizations were considered for example in~\cite{AlloucheScheicherTichy2000,
Charlier2018,
CharlierRampersadShallit2012,
LeroyRigoStipulanti2017,
RigoMaes2002}. In this paper, we focus on the class of Pisot numeration systems. A Pisot number is an algebraic integer greater than $1$ all of whose Galois conjugates have modulus less than $1$. In these systems, natural numbers are expanded via an increasing base sequence $U=(U(i))_{i\in\N}$ with bounded quotients $\frac{U(i+1)}{U(i)}$ and that satisfies a linear recurrence relation whose characteristic polynomial is the minimal polynomial of a Pisot number. We then talk about $(U,\K)$-regular sequences (precise definitions will be given in Section~\ref{Section : Preliminaries}). We refer to~\cite{FrougnySakatovitch2010} for a survey on Pisot numeration systems.

In the case of integer bases, it is not hard to see that a sequence $f\colon\N\to\K$ is $(b,\K)$-regular if and only if the series $\sum_{w \in \{0,1,\ldots,b-1\}^*} f(\val_b(w))\, w$ is $\K$-recognizable \cite{BerstelReutenauer2011}. Proving this equivalence consists in dealing with non-greedy representations of integers. For integer base numeration systems, this boils down to adding or removing leading zeroes. However, when considering more general numeration systems, there exist non-greedy representations of another kind. For example, the words $11$ and $100$ both represent the integer $3$ in the Zeckendorf numeration system, which is based on the sequence $(1,2,3,5,8,13,\ldots)$ of Fibonacci numbers \cite{Zeckendorf1972}. Our aim is to provide characterizations of regular sequences in the extended framework of Pisot numeration systems that only depend on the value function. In other words, we want to characterize the family of $(U,\K)$-regular sequences in a way that does not depend on the choice of an algorithm to represent numbers.
Moreover, we will carry our work in a multidimensional setting, meaning that we will define and study multidimensional $(\bU,\K)$-regular sequences $f\colon\N^d\to \K$ where $\bU=(U_1,\ldots,U_d)$ is a $d$-tuple of Pisot numeration systems. 

Our method relies on the property that normalization in Pisot numeration systems is realizable by a computable finite automaton~\cite{BruyereHansel1997,FrougnySolomyak1996}.
The existence of such a {\em normalizer} will be crucial in the proof of our main result. Along the way, we will show the existence of normalizers extended both to negative values and to the multidimensional setting. Let us also stress that all our techniques are effective, meaning that linear representations are computable for all formal series that will be proven to be $\K$-recognizable.

The paper is organized as follows. In Section~\ref{Section : Preliminaries}, we provide the necessary background. First, we define $\K$-automata and formal series. We then introduce the notions of $\K$-recognizable series and numeration systems, in particular Pisot numeration systems. In Section~\ref{Section : Multi (U,K)-regular sequences}, we define multidimensional $(\bU,\K)$-regular sequences and we state our main result, which is Theorem~\ref{Thm : EquivU}. It provides characterizations of multidimensional $(\bU,\K)$-regular sequences that are independent of the choice of admissible $\bU$-representations of integers.  Section~\ref{Section : Normalization} is concerned with normalization in Pisot numeration systems. We first recall the existence of a normalizer associated with any Pisot numeration system in dimension $1$, that is, a $2$-tape automaton accepting words of the form $\couple{u}{v}$ where $u$ has a non-negative value $n$ and $v$ is the greedy representation of $n$. We then extend the normalizer to words with negative values and to the multidimensional setting. Finally, in Section~\ref{Section : Generalization}, we prove Theorem~\ref{Thm : EquivU}. In the process, we define and use a particular operation mixing $2d$-tape automata and $\K$-automata. In Section~\ref{Sec:RepLin}, we end by making the linear representations involved in the proof of Theorem~\ref{Thm : EquivU} explicit. Throughout the paper, we use a running example in order to illustrate the definitions as well as the construction in the proof of Theorem~\ref{Thm : EquivU}.

\section{Preliminaries}
\label{Section : Preliminaries}

We make use of common notions in formal language theory, such as alphabet, letter, word, length of a word, language and usual definitions from automata theory \cite{Lothaire2002}. In particular, we let $\varepsilon$ denote the empty word. 

We take the convention that elements of $\Z^d$ are written as vertical vectors. Moreover, a word 
\[
\left(\begin{smallmatrix}
a_{i_{1,1}} \\
a_{i_{2,1}} \\
\vdots\\
a_{i_{d,1}}
\end{smallmatrix}\right)
\left(\begin{smallmatrix}
a_{i_{1,2}} \\
a_{i_{2,2}} \\
\vdots\\
a_{i_{d,2}}
\end{smallmatrix}\right)
\cdots
\left(\begin{smallmatrix}
a_{i_{1,\ell}} \\
a_{i_{2,\ell}} \\
\vdots\\
a_{i_{d,\ell}}
\end{smallmatrix}\right)
\] 
in $(\Z^d)^*$ will be written as
\[
\left(\begin{smallmatrix}
a_{i_{1,1}}a_{i_{1,2}} \cdots a_{i_{1,\ell}}\\
a_{i_{2,1}}a_{i_{2,2}} \cdots a_{i_{2,\ell}}\\
\vdots\\
a_{i_{d,1}}a_{i_{d,2}} \cdots a_{i_{d,\ell}}
\end{smallmatrix}\right).
\]
In what follows, it is important to observe that if a word \[
	\duple{w_1}{w_d}
\] 
belongs to $(\Z^d)^*$ then necessarily $w_1,\ldots,w_d$ are words of equal lengths. 

We use bold letters to designate $d$-dimensional elements. Most of the time, $\bn$ is a vector in $\Z^d$, $\ba,\bb$ are letters in $\Z^d$ and $\bu,\bv,\bw$ are words in $(\Z^d)^*$. Moreover, $\bz$ denotes the zero letter of dimension $d$ and $\boldsymbol{\varepsilon}$ denotes the empty word of dimension $d$. Throughout the text, the context will clearly indicate if an element of $\Z^d$ should be seen as a vector of integers or as a letter. 

In order to avoid any confusion between the notions of length and absolute value, we use the notation $||\cdot||$ in order to designate the component-wise absolute value of vectors of integers given by
\[
	\Bigg|\Bigg|\duple{n_1}{n_d}\Bigg|\Bigg|
	=\duple{||n_1||}{||n_d||}.
\]
We extend this notation to words over $\Z^d$ by setting $||\boldsymbol{\varepsilon}||=\boldsymbol{\varepsilon}$ and for all $\bu,\bv\in(\Z^d)^*$, $||\bu\bv||=||\bu||\, ||\bv||$. The length of a word $w$ is, as usual, written $|w|$.

From now on, let $\K$ be a fixed semiring and let $d$ be a dimension, that is, a positive integer.

\subsection{Weighted automata}
\label{Section : Weighted finite automata}

A \emph{weighted automaton} $\mathcal{A}=(Q,I,T,A,E)$ with weights in $\K$, or simply a \emph{$\K$-automaton}, is composed of a finite set $Q$ of states, a finite alphabet $A$ and of three mappings $I\colon Q \to \K$, $T\colon Q \to \K$ and $E\colon Q \times A \times Q \to \K$. A \emph{transition} $(p,a,q)\in Q \times A \times Q$ such that $E(p,a,q)\ne 0$ is an \emph{edge}, the letter $a$ is its \emph{label} and $E(p,a,q)$ is its \emph{weight}. With each $q\in Q$, we associate an \emph{initial weight} $I(q)$ and a \emph{final weight} $T(q)$. We call a state $q\in Q$ \emph{initial} (resp.\ \emph{final}) if $I(q) \neq 0$ (resp.\ $T(q) \neq 0$). A \emph{path} in $\mathcal{A}$ is a sequence $c=(q_0,a_1,q_1)(q_1,a_2,q_2) \cdots (q_{n-1},a_n,q_n)$ of consecutive edges. For each path $c$, we let $i_c$ and $t_c$ denote the first and last states $q_0$ and $q_n$ of $c$ respectively. The \emph{weight} of the path $c$ is the product $E(c)=E(q_0,a_1,q_1)E(q_1,a_2,q_2)\cdots E(q_{n-1},a_n,q_n)$ of the weights of its edges and its \emph{label} is the word $a_1a_2\cdots a_n$. For each $w\in A^*$, we let $C_\mathcal{A}(w)$ denote the set of paths in $\mathcal{A}$ of label $w$.
The \emph{weight} of a word $w\in A^*$ in $\mathcal{A}$ is the quantity $\sum_{c\in C_\mathcal{A}(w)} I(i_c)E(c)T(t_c)$.

A $\K$-automaton is represented by a graph. Each state is a vertex and each edge carries an expression of the form $a|k$, where $a$ is its label and $k$ is its weight. Each initial (resp.\ final) state $q$ is distinguished by an incoming (resp.\ outgoing) arrow which carries the weight $I(q)$ (resp.\ $T(q)$). 

\begin{example}
We start a running example by considering the $\N$-automaton over the alphabet $\{0,1,2\}$ depicted in Figure~\ref{Fig : AutPoidsSerie}.
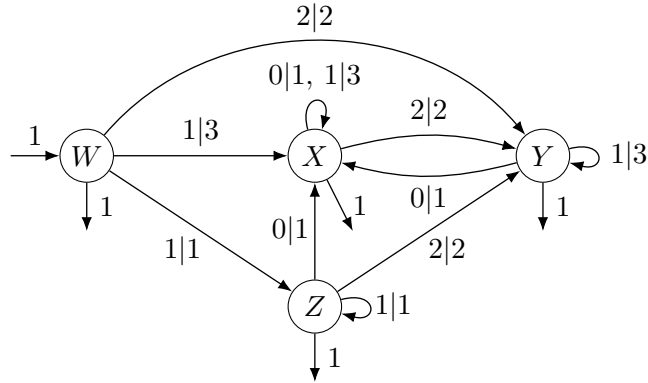
\begin{figure}[htb]
\begin{center}
\begin{tikzpicture}
\tikzstyle{every node}=[shape=circle, fill=none, draw=black,
minimum size=20pt, inner sep=2pt]
\node(1) at (0,0) {$W$};
\node(2) at (3,0) {$X$};
\node(3) at (6,0) {$Y$};
\node(4) at (3,-2) {$Z$};
\tikzstyle{every node}=[shape=circle, fill=none, draw=black, minimum size=15pt, inner sep=2pt]
\tikzstyle{every path}=[color=black, line width=0.5 pt]
\tikzstyle{every node}=[shape=circle, minimum size=5pt, inner sep=2pt]
\draw [-Latex] (-1,0) to node [above] {$1$} (1); 
\draw [-Latex] (1) to node [right] {$1$} (0,-1); 
\draw [-Latex] (2) to node [right] {$1$} (3.5,-1);
\draw [-Latex] (3) to node [right] {$1$} (6,-1); 
\draw [-Latex] (4) to node [right] {$1$} (3,-3); 

\draw [-Latex] (1) to node [above=-0.1] {$1|3$} (2);
\draw [-Latex] (1) to [pos=0.4] node [below] {$1|1$} (4);
\draw [-Latex] (1) [bend left=50] to node [above=-0.1] {$2|2$} (3);
\draw [-Latex] (2) to [loop above,looseness=7] node [above=-0.4] {$0|1,\, 1|3$} (2);
\draw [-Latex] (2) [bend left=15] to node [above=-0.1] {$2|2$} (3);
\draw [-Latex] (3) [bend left=15] to node [below=-0.1] {$0|1$} (2);
\draw [-Latex] (3) to [loop right,looseness=7] node [right] {$ 1|3$} (3);
\draw [-Latex] (4) to [loop right,looseness=7] node [right=-0.1] {$1|1$} (4);
\draw [-Latex] (4) to node [left=-0.1] {$0|1$} (2);
\draw [-Latex] (4) to node [pos=0.6,below] {$2|2$} (3);
\end{tikzpicture}
\end{center}
\caption{An $\N$-automaton over the alphabet $\{0,1,2\}$.}
\label{Fig : AutPoidsSerie}
\end{figure}
\end{example}

\subsection{Formal series}
\label{Section : Formal series}

A \emph{formal series} over a finite alphabet $A$ is a function $S\colon A^* \to \K$. The image under $S$ of a word $w$ is denoted by $(S, w)$ and is called the \emph{coefficient} of $w$ in $S$. We use the notation $S = \sum_{w\in A^*} (S,w)\, w$. In what follows, we simply talk about \emph{series} instead of formal series.

The \emph{Hadamard product} of two series $S,T \colon A^* \to \K$ is the series $S \odot T$ defined by $\sum_{w\in A^*} (S,w) (T,w)\; w$. If $L$ is a language over $A$, then its \emph{characteristic series} is the series $\underline{L} = \sum_{w \in L} w$. In particular, for $S\colon A^* \to \K$ and $L \subseteq A^*$, we have $ S \odot \underline{L} = \sum_{w\in L} (S,w) \, w$.

A series $S \colon A^* \to \K$ is \emph{$\K$-recognizable} if there exist an integer $r \in\N_{\ge 1}$, a morphism of monoids $\mu \colon A^* \to \K^{r\times r}$ (with respect to concatenation of words and multiplication of matrices) and two matrices $\lambda \in \K^{1 \times r}$ and $\gamma \in \K^{r \times 1}$ such that, for all words $w\in A^*$, $(S, w) = \lambda \mu(w) \gamma$. In this case, the triple $(\lambda, \mu, \gamma)$ is called a \emph{linear representation} of $S$.

Formal series and $\K$-automata can be linked. 
Given a $\K$-automaton $\mathcal{A}=(Q,I,T,A,E)$, the series \emph{recognized} by $\mathcal{A}$ is the series $\sum_{w\in A^*}\big(\sum_{c\in C_{\mathcal{A}}(w)} I(i_c)E(c)T(t_c)\big)\,w$ whose coefficients are the weights of the words in $\mathcal{A}$. 

\begin{proposition}
\label{Prop : LinkSeriesKAut}
A series is recognized by a $\K$-automaton if and only if it is $\K$-recognizable.
\end{proposition}

A constructive proof can be found in~\cite[Chapter 1, Proposition 6.1]{BerstelReutenauer2011}. Roughly, the linear representation encodes the weights in the $\K$-automaton.

\begin{example}
\label{Ex : SerieS}
For all $a\in \{0,1,2\}$ and all $w\in \{0,1,2\}^*$, let $|w|_a$ be the number of occurrences of the letter $a$ in the word $w$. Let $g\colon \{0,1,2\}^*\to \N$ be the function mapping a word $w$ to the greatest $n$ such that $1^n$ is a prefix of $w$.
Consider the regular language $L$ of the words over $\{0,1,2\}$ starting with $1$ or $2$ and avoiding factors in $21^*2$, and consider the series $S$ defined by
\begin{equation}
\label{Eq : series-ex}
	(S,w)=
	\begin{cases}
	3^{|w|_1}2^{|w|_2}+3^{|w|_1-g(w)}2^{|w|_2} 
	&\text{if }w\in L \text{ and } g(w)\ne 0\\
	3^{|w|_1}2^{|w|_2}	
	&\text{if }w\in L \text{ and } g(w)= 0\\
	0 
	&\text{if }w\notin L
	\end{cases}
\end{equation}
for all $w\in \{0,1,2\}^*$. For example, the coefficients of $1121$, $2101$ and $2112$ in $S$ are $3^3\cdot 2+3 \cdot 2=60$, $3^2\cdot 2=18$ and $0$ respectively. The series $S$ is $\N$-recognizable as a linear representation is given by $(\lambda, \mu, \gamma)$ where 
\[
\lambda=\left(\begin{smallmatrix}
1 & 0 & 0 & 0
\end{smallmatrix}\right),\
\mu(0)=\left(\begin{smallmatrix}
0 & 0 & 0 & 0\\
0 & 1 & 0 & 0\\
0 & 1 & 0 & 0\\
0 & 1 & 0 & 0
\end{smallmatrix}\right), \
\mu(1)=\left(\begin{smallmatrix}
0 & 3 & 0 & 1\\
0 & 3 & 0 & 0\\
0 & 0 & 3 & 0\\
0 & 0 & 0 & 1
\end{smallmatrix}\right), \
\mu(2)=\left(\begin{smallmatrix}
0 & 0 & 2 & 0\\
0 & 0 & 2 & 0\\
0 & 0 & 0 & 0\\
0 & 0 & 2 & 0
\end{smallmatrix}\right),\
\gamma=\left(\begin{smallmatrix}
1\\
1\\
1\\
1
\end{smallmatrix}\right).
\]
The corresponding $\N$-automaton, with respect to the order of the states $W,X,Y,Z$, is that of Figure~\ref{Fig : AutPoidsSerie}. 
\end{example}

\subsection{Numeration systems}\label{Sec : NumSys}
A \emph{numeration system} is given by an increasing sequence $U=(U(i))_{i\in\N}$ of integers such that $U(0) = 1$ and $\sup_{i\in\N} \frac{U(i+1)}{U(i)}<+\infty$. The \emph{value function} $\val_U \colon \Z^* \to \Z$ maps any word $w = a_\ell \cdots a_0 \in \Z^*$ to $\val_U(w) = \sum_{i=0}^\ell a_i U(i)$. A \emph{$U$-representation} of an integer $n$ is a word $w\in \Z^*$ such that $n=\val_U(w)$. The \emph{greedy $U$-representation} $\rep_U(n)$ of a positive integer $n$ is the unique $U$-representation $a_\ell \cdots a_0$ of $n$ such that  for all $j \in \{0,\ldots,\ell\}$, $a_j\ge 0$ and $\sum_{i=0}^j a_iU(i)< U(j+1)$, and moreover $a_\ell \neq 0$. It implies that the digits of greedy $U$-representations belong to the alphabet $A_U=\{0, \ldots,  \sup_{i\in\N} \big\lceil \frac{U(i+1)}{U(i)}\big\rceil - 1 \}$. We set $\rep_U(0) = \varepsilon$. The \emph{numeration language} is the set $\rep_U(\N)=\{\rep_U(n)\colon n\in\N\}$.

A numeration system is \emph{linear} if it satisfies a linear recurrence relation over $\Z$. A \emph{Pisot number} is an algebraic integer greater than $1$ whose Galois conjugates all have modulus less than $1$. A \emph{Pisot numeration system} is a linear numeration system whose characteristic polynomial is the minimal polynomial of a Pisot number. For such a numeration system, $\rep_U (\N)$ is a regular language~\cite{FrougnySolomyak1996}. Note that Pisot numeration systems associated with the same Pisot number may only differ by their initial conditions.

\begin{example}
\label{Ex : SystemUphicarre}
Consider the Pisot number $\frac{3+\sqrt{5}}{2}$, which is the square of the golden ratio. Its minimal polynomial is $X^2-3X+1$. Let $U$ be the Pisot numeration system defined by $U(0)=1$, $U(1)=3$ and for $i\ge 2$, $U(i)=3U(i-1)-U(i-2)$. This system is the Bertrand numeration system associated with $\frac{3+\sqrt{5}}{2}$; see~\cite{Bertrand-Mathis1989,BruyereHansel1997}. We have $A_U=\{0,1,2\}$. A DFA accepting the language $0^*\rep_U(\N)$ is depicted in Figure~\ref{Fig : A-PhiCarre}. Observe that the language $L$ considered in Example~\ref{Ex : SerieS} is precisely the numeration language $\rep_U(\N)$.
\begin{figure}[htb]
\begin{center}
\begin{tikzpicture}
\tikzstyle{every node}=[shape=circle, fill=none, draw=black,
minimum size=20pt, inner sep=2pt]
\node(1) at (0,0) {$ $};
\node(2) at (2,0) {$ $};
\tikzstyle{every node}=[shape=circle, fill=none, draw=black,
minimum size=15pt, inner sep=2pt]
\node(Af) at (0,0) {};
\node(Bf) at (2,0) {};
\tikzstyle{every path}=[color=black, line width=0.5 pt]
\tikzstyle{every node}=[shape=circle, minimum size=5pt, inner sep=2pt]
\draw [-Latex] (-1,0) to node {} (1); 
\draw [-Latex] (1) to [loop above] node [above=-0.2] {$0,1$} (1);
\draw [-Latex] (2) to [loop above] node [above] {$1$} (2);
\draw [-Latex] (1) to [bend left] node [above]
{$2$} (2);
\draw [-Latex] (2) to [bend left] node [below]
{$0$} (1);
\end{tikzpicture}
\end{center}
\caption{An automaton accepting the language $0^*\rep_U(\N)$ where $U$ is the Bertrand numeration system associated with the Pisot number $\frac{3+\sqrt{5}}{2}$.}
\label{Fig : A-PhiCarre}
\end{figure}
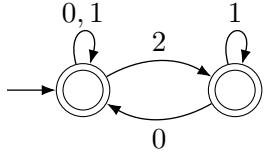
\end{example}

In order to be able to represent $d$-dimensional vectors of integers by words, we introduce the following usual convention. For a $d$-tuple of words
\[
	\duple{w_1}{w_d}
\]
in $(\Z^*)^d$, we set 
\[
	\duple{w_1}{w_d}^0
	=\duple{0^{\ell-|w_1|}w_1}{0^{\ell-|w_d|}w_d}
\]
where $\ell$ is the maximum of the lengths of the words $w_1,\ldots,w_d$. Thus, we have padded the shortest words with leading zeroes in order to obtain $d$ words of the same length. The obtained $d$-tuple can now be seen as a word over the $d$-dimensional alphabet $\Z^d$, that is, an element of $(\Z^d)^*$. For a $d$-tuple of numeration systems $\bU=(U_1,\ldots,U_d)$, we set $\bA_{\bU}=\dcart{A_{U_1}}{A_{U_d}}$. Accordingly, we extend the definition of the maps $\rep_U$ and $\val_U$ as follows:
\[
	\rep_{\bU}\colon \N^d \to A_{\bU}^*,\ 
	\duple{n_1}{n_d}
	\mapsto 
	\duple{\rep_{U_1}(n_1)}{\rep_{U_d}(n_d)}^0
\]
and 
\[
	\val_{\bU} \colon (\Z^*)^d\to \Z^d,\
	\duple{w_1}{w_d}
	\mapsto
	\duple{\val_{U_1}(w_1)}{\val_{U_d}(w_d)}.
\]
For an easier writing, for all $a\in\Z$ (considered as a letter), we write $\overline{a}$ instead of $-a$, and we extend this notation to words over $\Z$ by setting $\overline{uv}=\overline{u}\,\overline{v}$ for $u,v\in \Z^*$. 

\begin{example}
In dimension $2$, we have $\rep_{(U,U)}\couple{5}{9}=\couple{12}{101}^0=\couple{012}{101}=\couple{0}{1}\couple{1}{0}\couple{2}{1}$ and $\val_{(U,U)}\couple{021}{0001}=\couple{7}{1}$.
Moreover, we have $\val_{(U,U)}\couple{\overline{3}1\overline{1}}{011}=\couple{-22}{4}$.
\end{example}

\section{Multidimensional regular sequences}\label{Section : Multi (U,K)-regular sequences}

\begin{definition}
\label{Def : regular}
Let $\bU=(U_1,\ldots,U_d)$ be a $d$-tuple of numeration systems. A sequence $f\colon\N^d\to\K$ is called \emph{$(\bU,\K)$-regular} if the series $\sum_{\bn \in \N^d} f(\bn)\,\rep_{\bU}(\bn)$ is $\K$-recognizable.
\end{definition}

\begin{example}
Since the series defined by \eqref{Eq : series-ex} is $\N$-recognizable, the sequence 
\begin{align}
\label{Eq : f-ex}
	f\colon\N\to \N,\  n\mapsto 
	\begin{cases}
	3^{|\rep_U(n)|_1}\, (1+ 3^{-g(\rep_U(n))})\, 2^{|\rep_U(n)|_2} 
	& \text{if } g(\rep_U(n))\neq0\\
	3^{|\rep_U(n)|_1}\, 2^{|\rep_U(n)|_2} 
	& \text{otherwise}
	\end{cases}
\end{align}
is $(U,\N)$-regular. 
\end{example}

The aim of this work is to prove the following result, which testifies to the robustness of the notion of $(\bU,\K)$-regular sequences.

\begin{theorem}
\label{Thm : EquivU}
For $f\colon\N^d\to\K$ and a $d$-tuple $\bU=(U_1,\ldots,U_d)$ of Pisot numeration systems, the following assertions are equivalent.
\begin{enumerate}
\item The sequence $f$ is $(\bU,\K)$-regular.
\item For all finite alphabets $\bA\subset \Z^d$, the series $\sum_{\bw\in\bA^*} f(||\val_{\bU}(\bw)||)\,\bw$ is $\K$-recognizable.
\item The series $\sum_{\bw\in\bA_{\bU}^*} f(\val_{\bU}(\bw))\,\bw$ is $\K$-recognizable.
\item There exists a $\K$-recognizable series $S \colon \bA_{\bU}^*\to \K$ such that for all $\bn\in \N^d$, $(S,\rep_{\bU}(\bn))=f(\bn)$.
\end{enumerate}
\end{theorem}

\section{Multidimensional normalization and delay}\label{Section : Normalization}

For a numeration system $U$, the \emph{normalization} $\nu_U$ is the partial function from $\Z^*$ to $A_U^*$ which maps a word $w\in \Z^*$ with $\val_U(w)\ge 0$ to the word $\rep_U (\val_U (w))$, see~\cite{FrougnySolomyak1996}. 

\begin{theorem}[\cite{FrougnySolomyak1996}]
\label{Thm : Norm-FrougnySolo}
For any Pisot numeration system $U$ and any finite alphabet $A\subset\Z$, there exists a computable DFA accepting the language $\{\couple{u}{v}\in (A\times A_U)^*\colon v\in 0^*\nu_U(u)\}$.
\end{theorem}

The previous result is usually referred to by saying that the normalization from the alphabet $A$ is computable by a finite automaton. A different construction of such a $2$-tape automaton can be found in~\cite{BruyereHansel1997}. 

\begin{example}
\label{Ex : ZeroConvNormPhiCarre}
The language $\{\couple{u}{v}\in (\{0,1,2\}^2)^*\colon v\in 0^*\nu_U(u)\}$ is accepted by the DFA depicted in Figure~\ref{Fig : NormaliserPhiCarre}. 
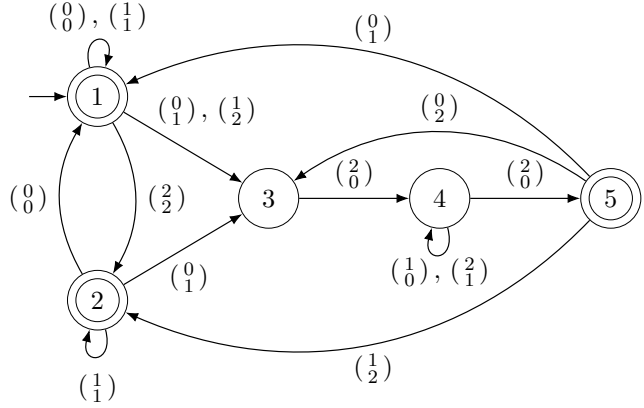
\begin{figure}[htb]
\begin{center}
\scalebox{0.9}{\begin{tikzpicture}
\tikzstyle{every node}=[-Latex, shape=circle, fill=none, draw=black,minimum size=25pt, inner sep=0pt]
\node(1) at (0,3) {$1$};
\node(2) at (0,0) {$2$};
\node(3) at (2.5,1.5) {$3$};
\node(4) at (5,1.5) {$4$};
\node(5) at (7.5,1.5) {$5$};
\tikzstyle{every node}=[shape=circle, fill=none, draw=black,
minimum size=18pt, inner sep=2pt]
\node(1f) at (0,0) {};
\node(2f) at (0,3) {};
\node(5f) at (7.5,1.5) {};
\tikzstyle{every path}=[color=black, line width=0.5 pt]
\tikzstyle{every node}=[shape=circle, minimum size=5pt, inner sep=2pt]
\draw [-Latex] (-1,3) to node {} (1); 
\draw [-Latex] (1) to [loop above,looseness=6] node [above=-0.5] {$\couple{0}{0},\, \couple{1}{1}$} (1);
\draw [-Latex] (2) to [loop below,looseness=6] node [below] {$ \couple{1}{1}$} (2);
\draw [-Latex] (1) to [bend left] node [right] {$\couple{2}{2}$} (2);
\draw [-Latex] (2) to [bend left] node [left] {$\couple{0}{0}$} (1);
\draw [-Latex] (1) to node [above=-0.1, pos=0.7] {$\couple{0}{1},\,\couple{1}{2}$} (3);
\draw [-Latex] (2) to node [below,pos=0.55] {$\couple{0}{1}$} (3);
\draw [-Latex] (3) to node [above=-0.1] {$\couple{2}{0}$} (4);
\draw [-Latex] (4) to node [above=-0.1] {$\couple{2}{0}$} (5);
\draw [-Latex] (4) to [loop below,looseness=6] node [below=-0.6] {$\couple{1}{0},\, \couple{2}{1}$} (4);
\draw [-Latex] (5) to [bend right=35] node [above=-0.1] {$\couple{0}{2}$} (3);
\draw [-Latex] (5) to [bend left=35] node [below=-0.1] {$\couple{1}{2}$} (2);
\draw [-Latex] (5) to [bend right=35] node [above=-0.1] {$\couple{0}{1}$} (1);
\end{tikzpicture}}
\end{center}
\caption{A DFA computing the normalization $\nu_U$ from the alphabet $\{0,1,2\}$ in the Bertrand numeration system $U$ associated with the Pisot number $\frac{3+\sqrt{5}}{2}$.}
\label{Fig : NormaliserPhiCarre}
\end{figure}
\end{example}

Let us extend the definition of the normalization in order to obtain a total function. We keep the same vocabulary and notation. For a numeration system $U$, we let $\overline{A_U}=\{\overline{a}\colon a\in A_U\}$ and we define the  \emph{(extended) normalization} as the function
\[
	\nu_U \colon \Z^* \to (A_U\cup \overline{A_U})^*,\ 
	w\mapsto 
	\begin{cases}
	\rep_U(\val_U(w)) 
	& \text{if } \val_U (w)\ge 0\\
	\overline{\rep_U(-\val_U(w))} 
	& \text{otherwise.} 
	\end{cases}
\]
Note that $A_U\cap\overline{A_U}=\{0\}$. Also note that $\nu_U(\Z^*)\subseteq A_U^*\cup (\overline{A_U})^*$.

\begin{example}
We have $\val_U(\overline{2}\overline{2})=-8$, so $\nu_U(\overline{2}\overline{2})=\overline{\rep_U(8)}=\overline{1}00$.
\end{example}

We generalize Theorem~\ref{Thm : Norm-FrougnySolo} to the normalization extended to negative values. 

\begin{theorem}
\label{The : Norm-With-Neg}
For any Pisot numeration system $U$ and any finite alphabet $A\subset\Z$, there exists a computable DFA accepting the language $\{\couple{u}{v}\in (A\times (A_U\cup \overline{A_U}))^* \colon v\in 0^*\nu_U(u)\}$.
\end{theorem}

\begin{proof}
Let $A'\subset \Z$ be a symmetric alphabet of the form $\{-d,\ldots,d\}$ containing $A$. By Theorem~\ref{Thm : Norm-FrougnySolo}, we can compute a DFA  $\mathcal{N}_1=(Q,i,T,A'\times A_U,\delta_1)$ accepting the language $\{\couple{u}{v}\in (A'\times A_U)^*\colon v\in 0^*\nu_U(u)\}$. Define another DFA $\mathcal{N}_2=(Q,i,T,A'\times \overline{A_U},\delta_2)$ where 
\[
	\delta_2 \colon Q \times (A'\times \overline{A_U}) \to Q,\ 
	(q,\couple{a}{b})\mapsto \delta_1(q,	\couple{\overline{a}}{\overline{b}}).
\]
By construction, $\mathcal{N}_2$ accepts the language $\{\couple{u}{v}\in (A'\times \overline{A_U})^*\colon v\in 0^*\nu_U(u)\}$. Therefore, we can compute a third DFA $\mathcal{N}$ accepting the union of the previous two languages, which is equal to $\{\couple{u}{v}\in (A'\times (A_U \cup \overline{A_U}))^* \colon v\in 0^*\nu_U(u)\}$. Finally, by deleting all transitions whose label has a first component not belonging to $A$, we obtain the desired DFA.
\end{proof}

The definition of the normalization can be generalized to the multidimensional setting by applying the normalization component-wise and padding with leading zeroes as defined in Section~\ref{Sec : NumSys}. For a $d$-tuple of numeration systems $\bU=(U_1,\ldots,U_d)$, we let $\overline{\bA_{\bU}}=\dcart{\overline{A_{U_1}}}{\overline{A_{U_d}}}$ and we define the \emph{($d$-dimensional extended) normalization} by
\[
	\nu_{\bU} \colon (\Z^d)^* \to (\bA_{\bU}\cup \overline{\bA_{\bU}})^*,\ 
	\duple{w_1}{w_d}
	\mapsto 
	\duple{\nu_{U_1}(w_1)}{\nu_{U_d}(w_d)}^0.
\]

\begin{example}
We have $
	\nu_{(U,U)}\couple{\overline{22}}{10}
	= \couple{\nu_U(\overline{22})}{\nu_U(10)}^0	
	= \couple{\overline{1}00}{10}^0
	= \couple{\overline{1}00}{010}
	= \couple{\overline{1}}{0}\couple{0}{1}\couple{0}{0}$.
Note that $\nu_{(U,U)}\couple{\overline{22}}{10}\notin \bA_{(U,U)}^*\cup(\overline{\bA_{(U,U)}})^*$ since $\val_U(\overline{22})<0$ while $\val_U(10)>0$.
\end{example}

Let us define an operation on $2$-tape automata in order to build multidimensional normalizers.

\begin{definition}
\label{Def : Multidim-Normalizer}
For each $j\in\{1,\ldots,d\}$, let $\mathcal{A}_j=(Q_j, i_j, T_j, A_i\times B_i, \delta_j)$ be a $2$-tape automaton. We define a $2d$-tape automaton 
$\bigotimes_{j=1}^d \mathcal{A}_j = (Q, i, T, \bA\times \bB, \delta)$ by 
\begin{itemize}
\item $Q= \dcart{Q_1}{Q_d}$
\item $i=(i_1,\ldots,i_d)$
\item $T = \dcart{T_1}{T_d}$
\item $\bA=\dcart{A_1}{A_d}$ and $\bB=\dcart{B_1}{B_d}$
\item $\delta\colon Q\times(\bA\times\bB)\to Q,\ 
	\left((q_1,\ldots,q_d),\couple{\ba}{\bb}\right) 		\mapsto 
	\left(
	\delta_1\left(q_1,\couple{a_1}{b_1}\right), 
	\ldots, 
	\delta_d\left(q_d,\couple{a_d}{b_d}\right) 			\right)$
where 
\[
	\ba= \left(\begin{smallmatrix}
		a_1 \\ \vdots\\ a_d
		\end{smallmatrix}\right) 
	\text{ and }
	\bb=\left(\begin{smallmatrix}
		b_1 \\ \vdots\\ b_d
		\end{smallmatrix}\right).
\]
\end{itemize}
\end{definition}

The automaton $\bigotimes_{j=1}^d \mathcal{A}_j$ can also be viewed as a $2$-tape automaton where each component would be a $d$-tuple. We will make no distinction between these two equivalent points of view.

The following result generalizes Theorem~\ref{The : Norm-With-Neg} to the multidimensional setting.

\begin{theorem}
\label{Theorem : MultidimNorm}
For any $d$-tuple of Pisot numeration systems $\bU=(U_1,\ldots,U_d)$ and any finite alphabet $\bA\subset\Z^d$, there exists a computable DFA accepting the language $\{\couple{\bu}{\bv}\in (\bA \times (\bA_{\bU}\cup\overline{\bA_{\bU}}))^*\colon \bv\in \bz^*\nu_{\bU}(\bu)\}$. 
\end{theorem}

\begin{proof}
Let $A_1,\ldots,A_d$ be finite alphabets of integers such that $\bA\subseteq \dcart{A_1}{A_d}$. By Theorem~\ref{The : Norm-With-Neg}, for each $j\in\{1,\ldots,d\}$, we can compute a DFA $\mathcal{N}_j$ accepting the language $\{\couple{u}{v}\in (A_j\times (A_{U_j} \cup \overline{A_{U_j}}))^*\colon v\in 0^*\nu_{U_j}(u)\}$. Then the DFA $\bigotimes_{j=1}^d \mathcal{N}_j$ accepts the language $\{\couple{\bu}{\bv}\in ((\dcart{A_1}{A_d}) \times (\bA_{\bU}\cup\overline{\bA_{\bU}}))^*\colon \bv\in \bz^*\nu_{\bU}(\bu)\}$. Indeed, let $\couple{\bu}{\bv}\in ((\dcart{A_1}{A_d}) \times (\bA_{\bU}\cup\overline{\bA_{\bU}}))^*$ and write
\[
	\bu=\duple{u_1}{u_d}
	\quad\text{and}\quad
	\bv=\duple{v_1}{v_d}.
\]
By definition, the $2d$-tape automaton $\bigotimes_{j=1}^d \mathcal{N}_j$ accepts $\couple{\bu}{\bv}$ if and only if for every $j\in\{1,\ldots,d\}$, $\mathcal{N}_j$ accepts $\couple{u_j}{v_j}$. Finally, we obtain the desired DFA by removing from $\bigotimes_{j=1}^d \mathcal{N}_j$ all transitions with a label $\couple{\ba}{\bb}$ whose first component $\ba$ does not belong to $\bA$.
\end{proof}

For a $d$-tuple of numeration systems $\bU=(U_1,\ldots,U_d)$ and a finite alphabet $\bA\subset\Z^d$ such that the language $\{\couple{\bu}{\bv}\in (\bA\times (\bA_{\bU} \cup \overline{\bA_{\bU}}))^* \colon \bv\in \bz^*\nu_{\bU}(\bu)\}$ is regular, we define the \emph{($d$-dimensional extended) normalizer} $\mathcal{N}_{\bU,\bA}$ to be the trim minimal automaton of this language. When the context is clear, we set 
\[
	\mathcal{N}_{\bU,\bA}
	=(Q_{\mathcal{N}}, 
	i_{\mathcal{N}}, 
	T_{\mathcal{N}}, 
	\bA\times (\bA_{\bU}\cup\overline{\bA_{\bU}}),
	\delta_{\mathcal{N}}).
\] 
Moreover, when $\bA=\bA_{\bU}$ we simply write $\mathcal{N}_{\bU}$ instead of $\mathcal{N}_{\bU,\bA_{\bU}}$.

The normalizer $\mathcal{N}_{\bU,\bA}$ is at the core of the reasoning leading to the proof of Theorem~\ref{Thm : EquivU}. Let us make two useful observations. First, we note that the initial state is also final since the empty word is accepted by $\mathcal{N}_{\bU,\bA}$. Second, the following lemma is a direct consequence of the uniqueness of normalized representations.

\begin{lemma}
\label{Lem : DeterministicNormaliser-Dimd}
For any two states $q$ and $q'$ of the normalizer $\mathcal{N}_{\bU,\bA}$ and any letter $\ba\in \bA$, there exists at most one letter $\bb\in \bA_{\bU}\cup\overline{\bA_{\bU}}$ such that $(q,\couple{\ba}{\bb},q')$ is a transition in $\mathcal{N}_{\bU,\bA}$.
\end{lemma}

\begin{example}
The DFA of Figure~\ref{Fig : NormaliserPhiCarre} is exactly the normalizer $\mathcal{N}_U$ for $d=1$. In $\mathcal{N}_U$, the word $22$ is not readable as a first component. Instead, the word $022$ is a possible first component because $|\nu_U(22)|-|22|=|100|-|22|=1$. Similarly, for $d=2$, the normalizer $\mathcal{N}_{(U,U)}$ accepts the word
\[
	\left(\begin{smallmatrix}
	022\\
	010\\
	100\\
	010
	\end{smallmatrix}\right).
\]
\end{example}

Motivated by the previous example, we introduce the following definition.

\begin{definition}
For all $\bw \in (\Z^d)^*$, the \emph{delay} of $\bw$ is the quantity $d_{\bU}(\bw)=|\nu_{\bU}(\bw)|-|\bw|$. 
\end{definition}

Note that for all 
\[
	\bw=\duple{w_1}{w_d}\in(\Z^d)^*,
\] 
we have $d_{\bU}(\bw)=\max\limits_{1 \le j \le d} d_{U_j}(w_j)$.

\begin{example}
We have $d_U(1\overline{2})=-1$ since $\nu_U(1\overline{2})=1$. Moreover, $d_U(22)=1$ and $d_U(10)=0$, so $d_{(U,U)}\couple{22}{10}=1$ and $d_{(U,U)}\couple{022}{010}=0$. 
\end{example}

\begin{remark}
\label{Rk : delay}
We analyze the behavior of the normalizer $\mathcal{N}_{\bU,\bA}$ on words whose first $d$ components are of the form $\bz^\ell\bw$ with $\ell\in\N$. Note that the delay $d_{\bU}(\bz^\ell\bw)$ is positive (resp.\ zero, negative) if $\ell<d_{\bU}(\bw)$ (resp.\ $\ell=d_{\bU}(\bw)$, $\ell>d_{\bU}(\bw)$). If $\ell\ge d_{\bU}(\bw)$, then there is a unique word accepted by $\mathcal{N}_{\bU,\bA}$ with first $d$ components $\bz^\ell\bw$, which is $\couple{\bz^\ell\bw}{\bz^{\ell-d_{\bU}(\bw)}\nu_{\bU}(\bw)}$. If $\ell< d_{\bU}(\bw)$, then $d_{\bU}(\bw)> 0$ and $\mathcal{N}_{\bU,\bA}$ accepts no word with first $d$ components $\bz^\ell\bw$. In particular, there is no accepted word with $\bw$ as first $d$ components if $d_{\bU}(\bw)> 0$. Note that in the case where the alphabet $\bA$ does not contain any negative digit (in any component), then $d_{\bU}(\bw)<0$ only when $\bw=\bz^{-d_{\bU}(\bw)}\bw'$ with $d_{\bU}(\bw')=0$.
\end{remark}

\begin{example}
Let $w_1=22$, $w_2=022$ and $w_3=0022$. We have $\nu_U(w_1)=\nu_U(w_2)=\nu_U(w_3)=100$. Therefore, $d_U(w_1)=1$, $d_U(w_2)=0$ and $d_U(w_3)=-1$. For all $\ell\in \N$, the normalizer $\mathcal{N}_U$ accepts the words $\couple{0^{\ell+1}w_1}{0^\ell 100}$, $\couple{0^\ell w_2}{0^\ell 100}$ and $\couple{0^\ell w_3}{0^{\ell+1} 100}$. 

Now, let $z_1=10$, $z_2=010$ and $z_3=0010$. Then $d_U(z_1)=0$, $d_U(z_2)=-1$ and $d_U(z_1)=-2$.  In dimension $2$, we have $d_{(U,U)}\couple{w_1}{z_1}=1$, $d_{(U,U)}\couple{w_2}{z_2}=0$, $d_{(U,U)}\couple{w_3}{z_3}=-1$. For all $\ell\in \N$, the normalizer $\mathcal{N}_{(U,U)}$ accepts the words 
\[
	\left(\begin{smallmatrix}
	0^{\ell+1}w_1\\
	0^{\ell+1}z_1\\
	0^{\ell}100\\
	0^{\ell+1}10
	\end{smallmatrix}\right),\
	\left(\begin{smallmatrix}
	0^{\ell}w_2\\
	0^{\ell}z_2\\
	0^{\ell}100\\
	0^{\ell+1}10
	\end{smallmatrix}\right)
	\quad\text{and}\quad
		\left(\begin{smallmatrix}
	0^{\ell}w_3\\
	0^{\ell}z_3\\
	0^{\ell+1}100\\
	0^{\ell+2}10
	\end{smallmatrix}\right).
\] 
\end{example}

\section{Characterizations of $(\bU,\K)$-regular sequences}
\label{Section : Generalization}

From now on, we consider a $d$-tuple of Pisot numeration systems $\bU=(U_1,\ldots,U_d)$, a $d$-dimensional sequence $f\colon\N^d\to \K$ and a finite alphabet $\bA\subset \Z^d$. In order to prove the implication $(1)\implies (2)$ of  Theorem~\ref{Thm : EquivU}, some intermediate results and definitions are needed. Without loss of generality, we assume that $\bz\in\bA$.

In what follows, we use the notation\footnote{The letters $G$ and $V$ stand for "Greedy" and "Value" respectively.}
\[
	G_f=\sum_{\bn\in\N^d} f(\bn)\,\rep_\bU(\bn) 		\quad \text{and} \quad 
	V_{f,\bA}=\sum_{\bw\in\bA^*} f(||\val_{\bU}(\bw)||)\, \bw.
\]
If $f$ is $(\bU,\K)$-regular, then by Proposition~\ref{Prop : LinkSeriesKAut}, there exists a $\K$-automaton recognizing the series $G_f$. In this case, we let
\[
	\mathcal{A}_G=(Q_G, I_G, T_G, \bA_{\bU}, E_G)
\] 
be such a $\K$-automaton. Without loss of generality, by adding a new state if needed, we assume that $\mathcal{A}_G$ has a unique initial state $i_G$, which has no incoming edge \cite{Eilenberg1974,Sakarovitch2009}. By definition of $G_f$, we have $(G_f,\bz^k\bw)=0$ for any $k>0$ and any word $\bw\in \bA_{\bU}^*$. Therefore, we make the additional assumption that the unique initial state of $\mathcal{A}_G$ 
has no outgoing edge of label $\bz$. 

In order to construct a $\K$-automaton recognizing $V_{f,\bA}$  starting from $\mathcal{A}_G$, we introduce an operation between a $2$-tape DFA and a $\K$-automaton. For any two finite alphabets $A,B$ and any coding $\sigma\colon B^*\to B^*$, that is, a letter-to-letter morphism, this operation applied on a $2$-tape automaton $\mathcal{A}$ over $A \times B$ and a $\K$-automaton $\mathcal{B}$ over $\sigma(B)$ produces a new $\K$-automaton $\mathcal{A}\circledast\mathcal{B}$ over $A$. Roughly speaking, for all words $u\in A^*$, if there exists a word $v\in B^*$ such that $\couple{u}{v}$ is accepted by $\mathcal{A}$ then the new $\K$-automaton $\mathcal{A}\circledast\mathcal{B}$ mimics the behavior of the $\K$-automaton $\mathcal{B}$ on $\sigma(v)$, that is, the weight of $u$ in $\mathcal{A}\circledast\mathcal{B}$ is the weight of $\sigma(v)$ in $\mathcal{B}$. Similar (but different) products of automata can be found in~\cite{Sakarovitch2009}. 

\begin{definition}
\label{Def : CCS}
Let $A$ and $B$ be two finite alphabets and let $\sigma\colon B^*\to B^*$ be a coding. Let $\mathcal{A}=(Q_\mathcal{A}, i_\mathcal{A}, T_\mathcal{A}, A\times B, \delta_\mathcal{A})$ be a DFA such that for any two states $q,q' \in Q_\mathcal{A}$ and any letter $a\in A$, there exists at most one letter $b\in B$ such that $\delta_\mathcal{A}(q,\couple{a}{b})=q'$. Let $\mathcal{B} = (Q_\mathcal{B}, I_\mathcal{B}, T_\mathcal{B}, \sigma(B), E_\mathcal{B})$ be a $\K$-automaton. With $\mathcal{A}$ and $\mathcal{B}$, we associate a new $\K$-automaton $\mathcal{A} \circledast \mathcal{B} = (Q, I, T, A, E)$ as follows. 
\begin{itemize}
\item $Q=  Q_\mathcal{A} \times Q_\mathcal{B}$.
\item $I\colon Q \to \K,\ (q,q')\mapsto 
\begin{cases}
	I_\mathcal{B}(q') &\text{if } q=i_\mathcal{A}\\
	0	&\text{else}.
\end{cases}$  
\item $T\colon Q \to \K,\ (q,q')\mapsto 
\begin{cases}
	T_\mathcal{B}(q') &\text{if } q\in T_{\mathcal{A}}\\
	0	&\text{else}.
\end{cases}$   
\item 
$E\colon Q \times A \times Q \to \K,\
((q_1,q'_1),a,(q_2,q_2'))\mapsto 
\begin{cases}
	E_\mathcal{B}(q_1',\sigma(b),q_2') &\text{if }\exists b\in B,\,  \delta_{\mathcal{A}}(q_1, \couple{a}{b})= q_2\\
	0	&\text{else}.
\end{cases}$
\end{itemize}
\end{definition}

Note that the extra assumption on $\mathcal{A}$ is required to ensure that $E$ is a well-defined function. An illustration of this operation is given by the black part of Figure~\ref{Fig : ProdSerieNormalisateur} in Example~\ref{Ex : ExempleEquiv}.

Consider the coding $\sigma\colon (\bA_{\bU} \cup \overline{\bA_{\bU}})^* \to \bA_{\bU}^*,\, \bw\mapsto||\bw||$. By Lemma~\ref{Lem : DeterministicNormaliser-Dimd}, the $\K$-automaton $\mathcal{N}_{\bU,\bA} \circledast \mathcal{A}_G$ (with respect to the coding $\sigma$) is well defined. Let $\mathcal{N}_{\bU,\bA}  \circledast \mathcal{A}_G= (Q, I, T, \bA, E)$. We now establish some properties of the latter $\K$-automaton. First, we note that $\mathcal{N}_{\bU,\bA}  \circledast \mathcal{A}_G$ has a unique initial state, which has no incoming edge. Then we prove a technical lemma.

\begin{lemma}
\label{Lem : unique-ell}
For all $q\in Q$, $\ba\in \bA$ and $\ell_1,\ell_2 \in\N$, if there exist two paths labeled by $\bz^{\ell_1} \ba$ and $\bz^{\ell_2}\ba$ from $(i_\mathcal{N},i_G)$ to $q$ in $\mathcal{N}_{\bU,\bA} \circledast \mathcal{A}_G$, then $\ell_1=\ell_2$.
\end{lemma}

\begin{proof}
Let $(p,p')\in Q$, $\ba\in \bA$ and $\ell_1,\ell_2 \in\N$ such that there exist two paths labeled by $\bz^{\ell_1} \ba$ and $\bz^{\ell_2} \ba$ from $(i_\mathcal{N},i_G)$ to $(p,p')$ in $\mathcal{N}_{\bU,\bA} \circledast \mathcal{A}_G$. Then there exist paths respectively labeled by $\couple{\bz^{\ell_1} \ba}{\bv_1}$ and $\couple{\bz^{\ell_2} \ba}{\bv_2}$ for some $\bv_1,\bv_2\in (\bA_{\bU} \cup \overline{\bA_{\bU}})^*$ from $i_\mathcal{N}$ to $p$ in $\mathcal{N}_{\bU,\bA}$. By hypothesis on $\mathcal{A}_G$, $\bv_1$ and $\bv_2$ cannot start with $\bz$. Since $\mathcal{N}_{\bU,\bA}$ is co-accessible, there exists a path, say of label $\couple{\mathbf{x}}{\mathbf{y}}$, from $p$ to a final state in $\mathcal{N}_{\bU,\bA}$. It follows that $\bv_1\mathbf{y}=\nu_U(\bz^{\ell_1} \ba\mathbf{x})=\nu_U(\bz^{\ell_2} \ba\mathbf{x})=\bv_2\mathbf{y}$ which implies $\bv_1=\bv_2$, and in turn $\ell_1=\ell_2$ as desired.
\end{proof}

Next, we study the series recognized by $\mathcal{N}_{\bU,\bA}  \circledast \mathcal{A}_G$.

\begin{lemma}
\label{Lem : ProduitCSSSerie-Prop}
Let $S$ be the series recognized by $\mathcal{N}_{\bU,\bA} \circledast \mathcal{A}_G$. For all $\bw\in \bA^*$ such that $d_{\bU}(\bw) \ne 0$, any path $c\in C_{\mathcal{N}_{\bU,\bA} \circledast \mathcal{A}_G}(\bw)$ is such that $I(i_c)=0$ or $T(t_c)=0$, and 
\[
	(S,\bw)=
	\begin{cases}
	(G_f,||\nu_{\bU}(\bw)||) 	
	& \text{if } d_{\bU}(\bw)=0\\
	0 				
	& \text{otherwise}.
\end{cases}
\]
\end{lemma}

\begin{proof}
Let $\bw\in \bA^*$. If $d_{\bU}(\bw)> 0$, then by Remark~\ref{Rk : delay},  the DFA $\mathcal{N}_{\bU,\bA}$ accepts no word with $\bw$ as first $d$ components. Now, assume that $d_{\bU}(\bw) < 0$. By Remark~\ref{Rk : delay}, the unique word $\bv$ such that $\mathcal{N}_{\bU,\bA}$ accepts $\couple{\bw}{\bv}$ is  $\bv=\bz^{-d_{\bU}(\bw)}\nu_U(\bw)$, which starts with $\bz$. Since we assumed that there is no outgoing edge labeled by $\bz$ from $i_G$, there is no path labeled by $\bw$ in $\mathcal{N}_{\bU,\bA} \circledast \mathcal{A}_G$ starting from $(i_\mathcal{N},i_G)$. Therefore, if $d_{\bU}(\bu)\ne 0$ then any path in $C_{\mathcal{N}_{\bU,\bA} \circledast \mathcal{A}_G}(\bw)$ is such that $I(i_c)=0$ or $T(t_c)=0$, and thus $(S,\bw)=0$. 

Suppose now that $d_{\bU}(\bw)= 0$. Write $\bw=\ba_1\cdots \ba_{\ell}$ and $\nu_{\bU}(\bw)=\bb_1 \cdots \bb_{\ell}$. There exists a unique sequence of states $q_0,\ldots,q_{\ell}$ in $Q_{\mathcal{N}}$ such that $q_0=i_{\mathcal{N}}$, $q_{\ell}\in T_\mathcal{N}$ and for all $i\in\{1,\ldots,\ell\}$, $\delta_\mathcal{N}(q_{i-1},\couple{\ba_i}{\bb_i})=q_i$. 
Let $c\in C_{\mathcal{N}_{\bU,\bA} \circledast \mathcal{A}_G}(\bw)$ be such that $i_c=(i_\mathcal{N},i_G)$. Then $c$ can be decomposed as $c=\big((q_0,q_0'),\ba_1,(q_1,q_1')\big)\cdots \big((q_{\ell-1},q_{\ell-1}'),\ba_{\ell},(q_{\ell},q_{\ell}')\big)$ where $q_0'=i_G$ and $q_1',\ldots,q_{\ell}'$ are some states in $Q_G$. Therefore, 
\[
	I(q_0,q_0')\,E(c)\, T(q_\ell,q_\ell')
	=I_G(i_G)
	\left(
	\prod_{i=1}^\ell E_G(q_{i-1}',||\bb_i||,q_i')
	\right)
	T_G(q_\ell').
\]
We obtain
\[
	(S,\bw) 	  
	=\sum_{q_1',\ldots,q_\ell'\in Q_G} 
	I_G(i_G)
	\left( 
	\prod_{i=1}^\ell E_G(q_{i-1}',||\bb_i||,q_i') 		\right)
	T_G(q_\ell')
	=(G_f,||\nu_U(\bw)||). 
\]
\end{proof}

We now modify $\mathcal{N}_{\bU,\bA}  \circledast \mathcal{A}_G$ to create a new $\K$-automaton $\mathcal{A}_{V,\bA}$, which will be proven to recognize the series $V_{f,\bA}$. We define the $\K$-automaton $\mathcal{A}_{V,\bA} = (Q_V, I_V, T_V, \bA, E_V)$ as follows.
\begin{itemize}
\item $Q_V= Q \cup \{\alpha\}$.
\item $I_V\colon Q_V \to \K$ is defined by $\restriction{I_V}{Q}=I$ and $I_V(\alpha)=I(i_\mathcal{N},i_G)$.
\item $T_V\colon Q_V \to \K$ is defined by $\restriction{T_V}{Q}=T$ and $T_V(\alpha)=0$.
\item $E_V \colon Q_V \times \bA \times Q_V \to \K$ is defined as follows.
\begin{enumerate}
\item For all $(q_1,q_1'),(q_2,q_2')\in Q$ and $\ba\in\bA$,
\[
E_V((q_1,q_1'),\ba,(q_2,q_2'))		
=\begin{cases}
	1	&\text{if }q_1'=q_2'= i_G\text{ and }\delta_{\mathcal{A}}(q_1,\couple{\ba}{\bz})=q_2 \\	
	E((q_1,q_1'),\ba,(q_2,q_2'))		
	&\text{else.}
\end{cases}
\]
\item\label{Def : EV-2} For all $q\in Q$ and $\ba\in\bA$, $E_V(\alpha,\ba,q)$ is equal to the sum of the weights $E(c)$ of all paths $c$ from $(i_\mathcal{N},i_G)$ to $q$ in $\mathcal{N}_{\bU,\bA} \circledast \mathcal{A}_G$ labeled by $\bz^\ell \ba$ for some $\ell\in\N_{\ge 1}$.
\item For all $q\in Q_V$ and all $\ba\in \bA$, $E_V(q,\ba,\alpha)=0$.
\end{enumerate}
\end{itemize}

Note that in Item~\ref{Def : EV-2}, if such an integer $\ell$ exists, then it is unique by Lemma~\ref{Lem : unique-ell}, and moreover, it is bounded by the number of states in $Q$. 

The modification $\mathcal{A}_{V,\bA}$ of $\mathcal{N}_{\bU,\bA}  \circledast \mathcal{A}_G$ is symbolically depicted in Figure~\ref{Fig : GeneralCase}, where wavy arrows represent paths whereas straight arrows represent edges. Note that if the alphabet $\bA$ does not contain any negative digit (on any component), then the red part of Figure~\ref{Fig : GeneralCase} is reduced to the loop on the initial state. Let us emphasize that the red part was indeed not present in $\mathcal{N}_{\bU,\bA} \circledast \mathcal{A}_G$ since we took care to assume that the state $i_G$ of the $\K$-automaton $\mathcal{A}_G$ had no incoming edge. In view of Lemma~\ref{Lem : ProduitCSSSerie-Prop}, the $\K$-automaton $\mathcal{N}_{\bU,\bA} \circledast \mathcal{A}_G$ produces the correct coefficients of $V_{f,\bA}$ only for words with zero delay. The role of the state $\alpha$ is to deal with words in $\bA^*$ with positive delays whereas the red edges are added to take into account words with negative delays. Note that $\alpha$ might have outgoing edges of label $\bz$. This happens whenever there exist words $\bw\in\bA^*$ having a delay greater than or equal to $2$.

\begin{figure}[htb]
\begin{center}
\scalebox{0.8}{\begin{tikzpicture}[>=triangle 60]
\draw[rounded corners] (-0.88,-3) rectangle (8,3);
\draw[rounded corners] (5.4,-2.5) rectangle (7.5,2.5);
\draw[rounded corners,dashed] (-6.5,-4) rectangle (9,5);

\node at (7.2,-3.5) {$\mathcal{N}_{\bU,\bA}  \circledast \mathcal{A}_G$};
\node at (8.5,-4.5) {$\mathcal{A}_{V,\bA}$};
\node at (6.5,-1.5) {$\vdots$};
\node at (6.5,1.5) {$\vdots$};
\node at (3.25,0) {$\vdots$};

\tikzstyle{every node}=[shape=circle,fill=none,draw=black,minimum size=40pt, inner sep=2pt]
\node(1) at (0,0) {$(i_\mathcal{N},i_G)$};
\node(2) at (6.5,0) {$q$};
\node[fill=cyan](3) at (1.7,4) {$\alpha$};
\node[red](4) at (-5,3) {$(q_1,i_G)$};
\node[red](5) at (-5,0) {$(q_2,i_G)$};
\node[red](6) at (-5,-3) {$(q_3,i_G)$};

\tikzstyle{every node}=[shape=circle,fill=none,draw=black,minimum size=15pt, inner sep=2pt]
\tikzstyle{every path}=[color=black,line width=0.5 pt]
\tikzstyle{every node}=[shape=circle,minimum size=5pt,inner sep=2pt]
\draw [-Latex] (0,1.5) to node [right] {$I(i_\mathcal{N},i_G)$} (1); 
\draw [-Latex] (1) to node [right] {$T(i_\mathcal{N},i_G)$} (0,-1.5); 
\draw [-Latex,red] (1) to [in=250,out=210, loop, looseness=8] node [below] {$\bz|1$} (1); 
\draw [-Latex,red] (1) to node [above=-0.3] {$\substack{\bb|1,\\ \delta_\mathcal{N}(i_\mathcal{N},\couple{\bb}{\bz})=q_1}$} (4); 
\draw [-Latex,red] (1) to  node [above=-0.8] {$\substack{ \boldsymbol{c}|1, \\ \delta_\mathcal{N}(i_\mathcal{N},\couple{\boldsymbol{c}}{\bz})=q_2}$} (5); 
\draw [-Latex,red] (5) to  node [right] {$\substack{ \boldsymbol{d}|1, \\  \delta_\mathcal{N}(q_2,\couple{\boldsymbol{d}}{\bz})=q_3}$} (6); 
\draw [-Latex] (0,1.5) to node [above] {} (1); 
\draw [-Latex,cyan] (-0.8,4) to node [above=-0.5] {\textcolor{cyan}{$I(i_\mathcal{N},i_G)$}} (3);
\draw [-Latex,cyan] (3) to node [pos=0.4,left] {$\ba\ | \sum_{i=1}^{n}k_i$} (2);

\draw [-Latex,decorate,decoration={snake,amplitude=.4mm,post length=1mm}] (1) to [bend left=20] node [above=-0.35] {$\bz^{\ell_{\ba,q}}\ba | k_1$} (2) ; 
\draw [-Latex,decorate,decoration={snake,amplitude=.4mm,post length=1mm}] (1) to [bend right=20] node [below=-0.35] {$\bz^{\ell_{\ba,q}}\ba | k_n$} (2) ; 
\end{tikzpicture}}
\end{center}
\caption{The automaton $\mathcal{A}_{V,\bA}$ is a modification of $\mathcal{N}_{\bU,\bA}  \circledast \mathcal{A}_G$ which recognizes the series $V_{f,\bA}$.
}
\label{Fig : GeneralCase}
\end{figure}
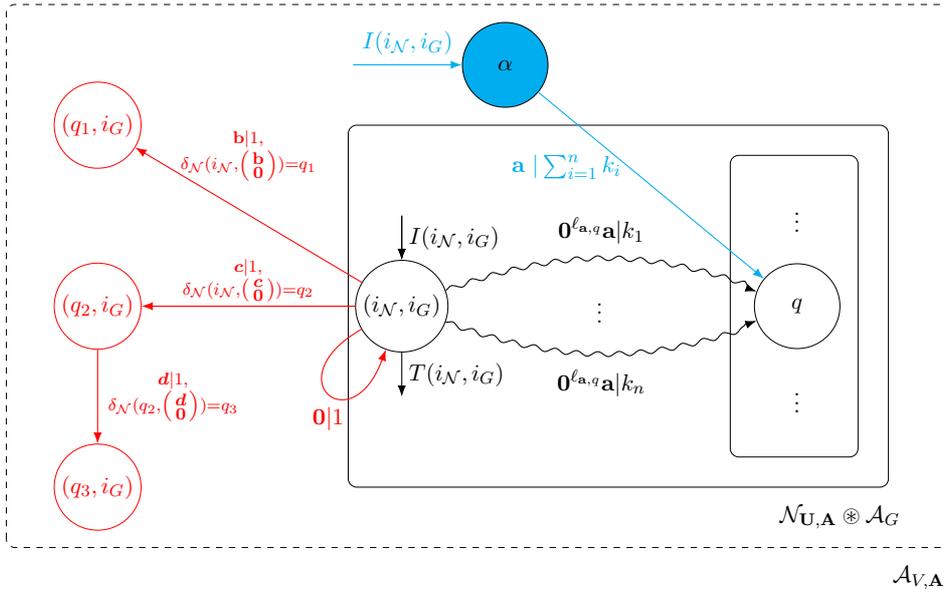

We are now ready to prove our main result. 

\begin{proof}[Proof of Theorem~\ref{Thm : EquivU}]
The implications $2 \implies 3$ and $3 \implies 4$ are clear. For any $\K$-recognizable series $S$ such that for all $\bn\in\N^d$, $(S,\rep_{\bU}(\bn))=f(\bn)$, we have $G_f=S\odot \underline{\rep_{\bU}(\N^d)}$. For every $j\in\{1,\ldots,d\}$, $\rep_{U_j}(\N)$ is a regular language since $U_j$ is a Pisot numeration system. Therefore, the language $\rep_{\bU}(\N^d)$ is regular as well. The implication $4 \implies 1$ now follows from \cite[Corollary 3.2.3]{BerstelReutenauer2011}.

It remains to prove $1\implies 2$. Suppose that $f\colon\N^d\to \K$ is $(\bU,\K)$-regular. The idea is to show that the modification $\mathcal{A}_{V,\bA}$ of $\mathcal{N}_{\bU,\bA} \circledast \mathcal{A}_G$ recognizes the series $V_{f,\bA}$. Let $S$ and $S'$ be the series recognized by the $\K$-automata $\mathcal{N}_{\bU,\bA} \circledast \mathcal{A}_G$ and $\mathcal{A}_{V,\bA}$ respectively. We have to show that $S'=V_{f,\bA}$. Let $\bw\in \bA^*$. Let $C_1$ be the set of those paths labeled by $\bw$ in $\mathcal{A}_{V,\bA}$ with a first edge of the form $((i_\mathcal{N},i_G),\ba,(q_1,i_G))$ and let $C_2$ be the set of those paths labeled by $\bw$ in $\mathcal{A}_{V,\bA}$ starting in the state $\alpha$. Note that each path in $C_{\mathcal{A}_V}(\bw) \setminus (C_1\cup C_2)$ is a path in $\mathcal{N}_{\bU,\bA} \circledast \mathcal{A}_G$. Therefore, $(S',\bw)$ is equal to
\begin{align}
\label{Eq : Thm : EquivU}
	(S,\bw) 
	+\sum_{c\in C_1} I_V(i_\mathcal{N},i_G)\, E_V(c)\, T_V(t_c)  
 	+\sum_{c \in C_2} I_V(\alpha)\, E_V(c)\, T_V(t_c).
\end{align}
In order to obtain that $(S',\bw)=(V_{f,\bA},\bw)$, we show that, depending on the delay $d_{\bU}(\bw)$, exactly two terms of the sum \eqref{Eq : Thm : EquivU} are zero while the third one is equal to $(V_{f,\bA},\bw)$. 

Let us write 
\[
	\bw=\duple{w_1}{w_d}.
\]
For all $j\in\{1,\ldots,d\}$, since $\nu_{U_j}(w_j)\in A_{U_j}^*\cup(\overline{A_{U_j}})^*$, we have that
$\val_{U_j}(||\nu_{U_j}(w_j)||)
	=||\val_{U_j}(\nu_{U_j}(w_j))||
	=||\val_{U_j}(w_j)||$.
Then by definition of the series $G_f$ and $V_{f,\bA}$, we get that 
\begin{equation}
\label{Eq : GV}
(G_f,||\nu_{\bU}(\bw)||)
=f(\val_{\bU}(||\nu_{\bU}(\bw)||)
= f(||\val_{\bU}(\bw)||) 
= (V_{f,\bA},\bw).
\end{equation} 
It then follows from Lemma~\ref{Lem : ProduitCSSSerie-Prop} that
\begin{equation}
\label{Eq : (S,w) et (V(f,A),w)}
(S,\bw)=
\begin{cases}
(V_{f,\bA},\bw) 	& \text{if } d_{\bU}(\bw)=0 \\
0 				& \text{else}.
\end{cases}
\end{equation}

Now, we consider the second term of \eqref{Eq : Thm : EquivU}. We show that 
\[
	\sum_{c\in C_1} I_V(i_\mathcal{N},i_G)\ E_V(c)\ T_V(t_c)  
 	=\begin{cases}
	(V_{f,\bA},\bw) 	& \text{if } d_{\bU}(\bw)<0 \\
	0 				& \text{else}.
	\end{cases}
\]
For $c\in C_1$, let $\ell_c\in \N_{\ge 1}$ be the number of times $c$ goes through states with $i_G$ as second component. By definition of $\mathcal{A}_{V,\bA}$, there exists a word $\bv_c\in (\bA_{\bU}\cup\overline{\bA_\bU})^*$ not starting with the letter $\bz$ such that $\couple{\bw}{\bz^{\ell_c}\bv_c}$ is the label of a path in $\mathcal{N}_{\bU,\bA}$ starting in the initial state $i_\mathcal{N}$. The latter path ends in a final state if and only if $d_{\bU}(\bw)=-\ell_c$ and $\bv_c=\nu_{\bU}(\bw)$. Therefore, if $d_{\bU}(\bw)\ge 0$ then $T_V(t_c)=0$. Now suppose that $d_{\bU}(\bw)< 0$. Let $\bw=\bw'\bw''$ with $|\bw'|=-d_{\bU}(\bw)$ and let $q_0,q_1,\ldots,q_{|\bw|}$ be the states visited along the unique accepting path labeled by $\couple{\bw}{\bz^{-d_{\bU}(\bw)}\nu_{\bU}(\bw)}$ in the DFA $\mathcal{N}_{\bU,\bA}$. Let $(C_1)''$ be the set of those paths $c''$ in $\mathcal{N}_{\bU,\bA} \circledast \mathcal{A}_G$ labeled by $\bw''=\ba_{|\bw'|}\cdots \ba_{|\bw|}$ of the form
\[
	c''=((q_{|\bw'|},i_G),\ba_{|\bw'|+1},(q_{|\bw'|+1},q_{|\bw'|+1}'))
	\cdots
	((q_{|\bw|-1},q_{|\bw|-1}'),\ba_{|\bw|},(q_{|\bw|},q_{|\bw|}'))\]
for some $q_{|\bw'|+1}',\ldots,q_{|\bw|}'\in Q_G$.
We get 
\begin{align*}
	\sum_{c\in C_1}I_V(i_\mathcal{N},i_G)E_V(c)\,  T_V(t_c)
	&=	\sum_{\substack{c\in C_1\\ \ell_c=-d_{\bU}(\bw)\\ \bv_c=\nu_{\bU}(\bw)}}I_V(i_\mathcal{N},i_G)E_V(c)\,  T_V(t_c)\\
	&=\sum_{c''\in (C_1)''}
	I(i_\mathcal{N},i_G)\, 1^{-d_{\bU}(\bw)}\, E(c'')T(t_{c''}).
\end{align*}
By a reasoning similar to that of the second part of the proof of Lemma~\ref{Lem : ProduitCSSSerie-Prop},	
the latter sum is equal to $(G_f,||\nu_{\bU}(\bw)||)$. We conclude by using~\eqref{Eq : GV}.

We turn to the third term of \eqref{Eq : Thm : EquivU}. We show that 
\[
 	\sum_{c\in C_2} I_V(\alpha)E_V(c)T_V(t_c)  
 	=\begin{cases}
	(V_{f,\bA},\bw) 	& \text{if } d_{\bU}(\bw)>0 \\
	0 				& \text{else}.
	\end{cases}
\]
If $\bw=\boldsymbol{\varepsilon}$, the desired equality follows from the facts that $d_{\bU}(\boldsymbol{\varepsilon})=0$ and $T_V(\alpha)=0$. We now suppose that $\bw\ne \boldsymbol{\varepsilon}$ and write $\bw=\ba\bw'$. Any path $c\in C_2$ can be decomposed as $c=(\alpha,\ba,q)c'$ where $q\in Q$ and $c'$ is path in $\mathcal{N}_{\bU,\bA} \circledast \mathcal{A}_G$ labeled by $\bw'$. For each $q\in Q$, we let $P_q$ be the set of those paths in $\mathcal{N}_{\bU,\bA} \circledast \mathcal{A}_G$ labeled by $\bw'$ and starting in $q$. By definition, the transition $(\alpha,\ba,q)$ has a non-zero weight in $\mathcal{A}_{V,\bA}$ if there exists a path in $\mathcal{N}_{\bU,\bA} \circledast \mathcal{A}_G$ labeled by $\bz^{\ell} \ba$ for some $\ell\in\N_{\ge 1}$ from $(i_\mathcal{N},i_G)$ to $q$. By Lemma~\ref{Lem : unique-ell}, all such paths have the same label $\bz^{\ell_{\ba,q}}\ba$. Then
\begin{align}
	\sum_{c\in C_2} I_V(\alpha)E_V(c)T_V(t_c) \nonumber 
	&= \sum_{q\in Q}\sum_{c'\in P_q} 
	I_V(\alpha)E_V(\alpha,\ba,q)E_V(c')T_V(t_{c'}) \nonumber \\
	&= \sum_{q\in Q} \sum_{c'\in P_q} I(i_\mathcal{N},i_G)  
	\left(
	\sum_{\substack{
	p\in C_{\mathcal{N}_{\bU,\bA} \circledast \mathcal{A}_G}(\bz^{\ell_{\ba,q}}\ba)\\
	i_p=(i_\mathcal{N},i_G),\, t_p=q}} 
	E(p)	\right) 
	E(c')T(t_{c'}) \nonumber \\
	&= \sum_{q\in Q} \
	\sum_{\substack{p\in C_{\mathcal{N}_{\bU,\bA} \circledast \mathcal{A}_G}(\bz^{\ell_{\ba,q}}\ba)\\i_p=(i_\mathcal{N},i_G),\, t_p=q}}\ 
	\sum_{c'\in P_q}  
	I(i_\mathcal{N},i_G)E(pc')T(t_{c'}).  
\label{Eq : c'} 
\end{align}
If $\ell_{\ba,q}\ne d_{\bU}(\bw)$, we have $d_{\bU}(\bz^{\ell_{\ba,q}}\bw)\ne 0$ and then it follows from Lemma~\ref{Lem : ProduitCSSSerie-Prop} that each path $pc'$ as in~\eqref{Eq : c'} is such that $T(t_{c'})=0$.
In particular, if $d_{\bU}(\bw)\le 0$ then each term of the sum is equal to $0$.
Now, we suppose that $d_{\bU}(\bw)> 0$.
Then the sum \eqref{Eq : c'} is equal to
\[
	\sum_{\substack{c\in C_{\mathcal{N}_{\bU,\bA} \circledast \mathcal{A}_G}(\bz^{d_{\bU}(\bw)}\bw)\\ i_c=(i_\mathcal{N},i_G)}}
	 I(i_\mathcal{N},i_G)E(c)T(t_c) 
	=(S,\bz^{d_{\bU}(\bw)}\bw).
\]
By~\eqref{Eq : (S,w) et (V(f,A),w)} applied to $\bz^{d_{\bU}(\bw)}\bw$, we have $(S,\bz^{d_{\bU}(\bw)}\bw)=(V_{f,\bA},\bz^{d_{\bU}(\bw)}\bw)=(V_{f,\bA},\bw)$, hence the conclusion.
\end{proof}

\begin{example}
\label{Ex : ExempleEquiv}
With our notation, if $f$ is the sequence \eqref{Eq : f-ex}, then the series $G_f$ is equal to \eqref{Eq : series-ex} and the series $V_{f,\{0,1,2\}}$ is given by
\[
	(V_{f,\{0,1,2\}},w)=
	\begin{cases}
	3^{|\nu_U(w)|_1}\, 2^{|\nu_U(w)|_2}
	+3^{|\nu_U(w)|_1-g(\nu_U(w))}\, 2^{|\nu_U(w)|_2}
	&\text{if } g(\nu_U(w))\ne 0\\
	3^{|\nu_U(w)|_1}\, 2^{|\nu_U(w)|_2}
	&\text{if } g(\nu_U(w))= 0,
	\end{cases}
\]
for all $w\in \{0,1,2\}^*$. Let us call $\mathcal{A}_G$ and $\mathcal{A}_{V,\{0,1,2\}}$ the $\K$-automata of Figures~\ref{Fig : AutPoidsSerie} and~\ref{Fig : ProdSerieNormalisateur} respectively. 
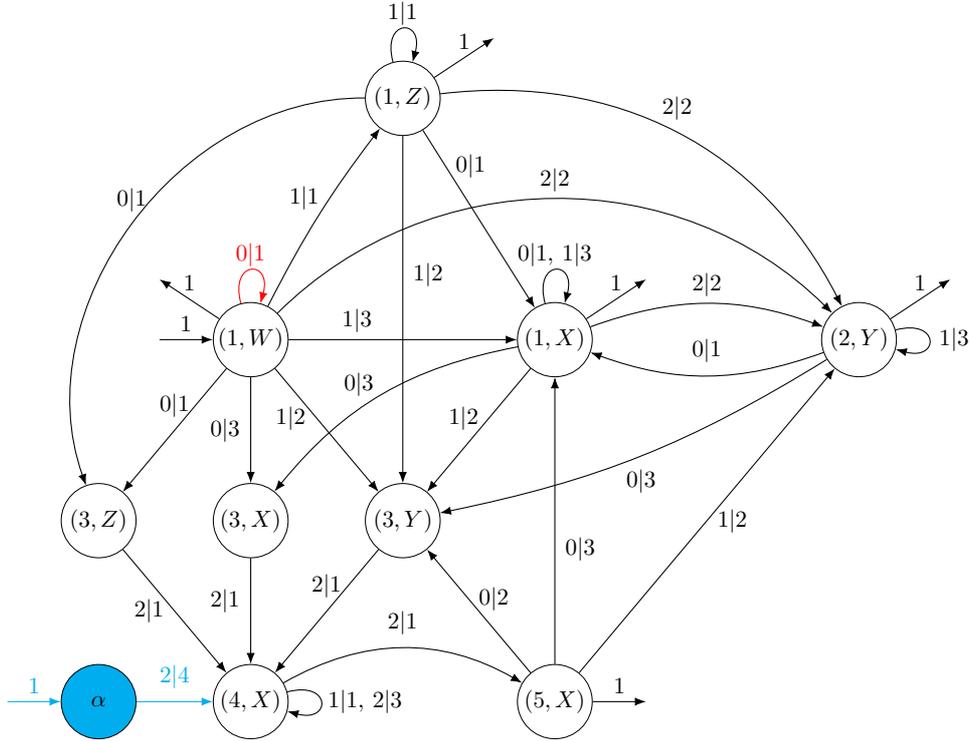
\begin{figure}[htb]
\begin{center}
\scalebox{0.8}{\begin{tikzpicture}
\tikzstyle{every node}=[-Latex, shape=circle, fill=none, draw=black,minimum size=35pt, inner sep=0pt]
\node(1) at (0,0) {$(1,W)$};
\node(2) at (5,0) {$(1,X)$};
\node(3) at (10,0) {$(2,Y)$};
\node(4) at (0,-3) {$(3,X)$};
\node(5) at (2.5,-3) {$(3,Y)$};
\node(6) at (0,-6) {$(4,X)$};
\node(7) at (5,-6) {$(5,X)$};
\node(8) at (2.5,4) {$(1,Z)$};
\node(9) at (-2.5,-3) {$(3,Z)$};
\node[fill=cyan](10) at (-2.5,-6) {$\alpha$};

\tikzstyle{every node}=[shape=circle, fill=none, draw=black,
minimum size=30pt, inner sep=2pt]
\tikzstyle{every path}=[color=black, line width=0.5 pt]
\tikzstyle{every node}=[shape=circle, minimum size=5pt, inner sep=2pt]

\draw [-Latex] (-1.5,0) to node [above] {$1$} (1); 
\draw [cyan,-Latex] (-4,-6) to node [above] {$1$} (10); 

\draw [-Latex] (1) to node [above] {$1$} (-1.5,1); 
\draw [-Latex] (2) to node [above] {$1$} (6.5,1); 
\draw [-Latex] (3) to node [above] {$1$} (11.5,1); 
\draw [-Latex] (7) to node [above] {$1$} (6.5,-6); 
\draw [-Latex] (8) to node [above] {$1$} (4,5);

\draw [red,-Latex] (1) to [loop above,looseness=6] node [above=-0.16] {$0|1$} (1);
\draw [-Latex] (1) to [pos=0.3] node [above=-0.1] {$1|3$} (2);
\draw [-Latex] (1) to [bend left=45] node [above=-0.1] {$2|2$} (3);
\draw [-Latex] (1) to [pos=0.5] node [left] {$0|3$} (4);
\draw [-Latex] (1) to [pos=0.4] node [left] {$1|2$} (5);
\draw [-Latex] (1) to [pos=0.6,bend left=5] node [left] {$1|1$} (8);
\draw [-Latex] (1) to node [above] {$0|1$} (9);

\draw [-Latex] (2) to [loop above,looseness=6] node [above=-0.48] {$0|1,\, 1|3$} (2);
\draw [-Latex] (2) to [bend left=20] node [above=-0.1] {$2|2$} (3);
\draw [-Latex] (2) to [pos=0.4] node [left] {$1|2$} (5);
\draw [-Latex] (2) to [pos=0.6,bend right=20] node [above] {$0|3$} (4);

\draw [-Latex] (3) to [loop right,looseness=6] node [right] {$1|3$} (3);
\draw [-Latex] (3) to [bend left=20] node [above] {$0|1$} (2);
\draw [-Latex] (3) to [bend left=10] node [below] {$0|3$} (5);

\draw [-Latex] (4) to node [pos=0.4,left] {$2|1$} (6);

\draw [-Latex] (5) to [pos=0.5] node [above] {$2|1$} (6);

\draw [-Latex] (6) to [loop right,looseness=6] node [right] {$1|1,\,2|3$} (6);
\draw [-Latex] (6) to [bend left] node [above] {$2|1$} (7);

\draw [-Latex] (7) to [] node [pos=0.4,right] {$0|3$} (2);
\draw [-Latex] (7) to [] node [pos=0.6,right] {$0|2$} (5);
\draw [-Latex] (7) to [] node [right] {$1|2$} (3);

\draw [-Latex] (8) to [loop above,looseness=6] node [above=-0.16] {$1|1$} (8);
\draw [-Latex] (8) to [pos=0.2] node [right] {$0|1$} (2);
\draw [-Latex] (8) to [bend left=35] node [above] {$2|2$} (3);
\draw [-Latex] (8) to [pos=0.4] node [right] {$1|2$} (5);
\draw [-Latex] (8) to [bend right=55] node [above] {$0|1$} (9);

\draw [-Latex] (9) to [] node [left] {$2|1$} (6);

\draw [cyan,-Latex] (10) to [] node [above] {$2|4$} (6);
\end{tikzpicture}}
\end{center}
\caption{The automaton $\mathcal{A}_{V,\{0,1,2\}}$ whose black part is $\mathcal{N}_U \circledast \mathcal{A}_G$, where $U$ is the Bertrand numeration system associated with the Pisot number $\frac{3+\sqrt{5}}{2}$.} 
\label{Fig : ProdSerieNormalisateur}
\end{figure}
Note that $\mathcal{A}_G$ satisfies our assumptions: it recognizes $G_f$ and has a unique initial state, which has no incoming edge and no outgoing edge of label $0$. By looking at the normalizer $\mathcal{N}_U$ in Figure~\ref{Fig : NormaliserPhiCarre}, we see that for all $w\in\{0,1,2\}^*$, the delay $d_{\bU}(w)$ is at most $1$. This implies that in $\mathcal{A}_{V,\{0,1,2\}}$, the state $\alpha$ has no outgoing edge labeled by $0$. By our construction, the series $V_{f,\{0,1,2\}}$ is recognized by $\mathcal{A}_{V,\{0,1,2\}}$, which is precisely the modified version of $\mathcal{N}_U \circledast \mathcal{A}_G$. 
\end{example}

\begin{remark}
In the proof of Theorem~\ref{Thm : EquivU}, the Pisot hypothesis might be too strong. For proving the implication $(1)\implies (2)$, it is merely used in order to get that the normalization is computable by a finite automaton. For the implication $(4)\implies (1)$, we only need the assumption that the numeration systems $U_j$ all have regular numeration languages $\rep_{U_j}(\N)$. The other implications do not need any assumption on the numeration systems. 
\end{remark}

\section{From a greedy-based linear representation to a value-based one}
\label{Sec:RepLin}

The proof of Theorem~\ref{Thm : EquivU} is constructive. In particular, by using the correspondence between $\K$-automata and linear representations, we provide an algorithm that, starting with a linear representation of the series $G_f=\sum_{\bn\in\N^d}f(\bn)\rep_{\bU}(\bn)$, computes a linear representation of the series $V_{f,\bA}=\sum_{\bw\in\bA^*} f(||\val_{\bU}(\bw)||)\,\bw$.

Let $(\lambda, \mu, \gamma)$ be a linear representation of $G_f$ with $\lambda \in \K^{1 \times r}$, $\mu \colon \bA_\bU^* \to \K^{r\times r}$ and $\gamma \in \K^{r \times 1}$ corresponding to the $\K$-automaton $\mathcal{A}_G$. In particular, $r$ is the number of states of $\mathcal{A}_G$ and for all $i,j\in\{1,\ldots,r\}$ and all $\ba\in\bA_{\bU}$, $\mu(\ba)_{i1}=0$, $\mu(\bz)_{1j}=0$ and $\lambda_i=0$ if $i\ne 1$. Using Definition~\ref{Def : CCS}, we first build a linear representation $(L,M,R)$ of the series recognized by $\mathcal{N}_{\bU,\bA} \circledast \mathcal{A}_G$ with $L \in \K^{1 \times rs}$, $M \colon \bA^* \to \K^{rs\times rs}$ and $R\in \K^{rs \times 1}$. Let $Q_{\mathcal{N}}=\{q_1,\ldots,q_s\}$ where $q_1$ is the unique initial state. We set 
\[
	L=	\begin{pmatrix}
	\lambda_1 & \undermat{rs-1}{0 & \cdots & 0} 
	\end{pmatrix}
	\quad\text{and}\quad
	R=\duple{R_1}{R_s}
\]
where $R_1,\ldots,R_s\in\K^{r\times 1}$ are given by
\[
	R_i =
	\begin{cases}
	\gamma 
	& \text{if } q_i\in T_\mathcal{N}\\
	0_{r\times 1} 
	& \text{otherwise}
	\end{cases}
\]
for all $i\in\{1,\ldots,s\}$. For all $\ba\in\bA$ and all $i,j\in\{1,\ldots,s\}$, we define matrices $M^{(a)}_{i,j}\in\K^{r\times r}$ by
\[
	M^{(a)}_{i,j}
	= \begin{cases}
	\mu(||\bb||) & \text{if }\exists \bb\in \bA_{\bU}\cup\overline{\bA_{\bU}},\, \delta_{\mathcal{N}}(q_i,\couple{\ba}{\bb})=q_j\\
0_{r\times r} & \text{otherwise}.
	\end{cases}
\]
Finally, for all $\ba\in\bA$, we set
\[
	M(\ba)
	=\begin{pmatrix}
	M^{(a)}_{1,1} & \cdots & M^{(a)}_{1,s}\\
	\vdots &  & \vdots\\
	M^{(a)}_{s,1} & \cdots & M^{(a)}_{s,s}\\
	\end{pmatrix}.
\]

The last step is to encode the modification $\mathcal{A}_{V,\bA}$ of $\mathcal{N}_{\bU,\bA} \circledast \mathcal{A}_G$ that consists in adding the state $\alpha$ together with its transitions and the so-called ``red edges''. A linear representation of the series $V_{f,\bA}$ is given by $(\lambda', \mu', \gamma')$ with $\lambda' \in \K^{1 \times (rs+1)}$, $\mu' \colon \bA^* \to \K^{(rs+1)\times (rs+1)}$ and $\gamma' \in \K^{(rs+1) \times 1}$ defined as follows.
We set $\lambda'=\left(\begin{smallmatrix} \lambda_1 & L \end{smallmatrix}\right)$ and $\gamma'=\couple{0}{R}$. For each $\ba\in\bA$ and each $j\in\{1,\ldots,rs\}$, we let $\ell_{\ba,j}$ be the unique positive $\ell$ such that there exists a path labeled by $0^\ell\ba$ from the initial state to the $j$-th state of $\mathcal{N}_{\bU,\bA} \circledast \mathcal{A}_G$ if it exists (see Lemma~\ref{Lem : unique-ell} and Figure~\ref{Fig : GeneralCase}). Here, we consider the order on the states of $\mathcal{N}_{\bU,\bA} \circledast \mathcal{A}_G$ induced by the matrix indexing. For $\ba\in\bA$, we define
\[
	\mu'(\ba)=
	\left(\begin{smallmatrix}
	0 &  m(\ba)\\
	0_{rs\times 1} & M'(\ba)
	\end{smallmatrix}\right)
\]
where for all $i,j\in\{1,\ldots,rs\}$,
\[
	(M'(\ba))_{i,j}
	=\begin{cases}
	1 				& \text{if } i\equiv j\equiv 1\bmod r\text{ and } \delta_{\mathcal{N}}\big(q_{\frac{i-1}{r}},\couple{\ba}{\bz}\big)=q_{\frac{j-1}{r}}
	\\
	(M(\ba))_{i,j} 	& \text{otherwise}
	\end{cases}
\]
and 
\[
	m(\ba)_j= 
	\begin{cases}
	M(\bz^{\ell_{\ba,j}}\ba)_{1,j} & \text{if } \ell_{\ba,j} \text{ exists} \\
	0	& \text{otherwise}. 
	\end{cases}
\]

To compute $m(\ba)$ for each $\ba\in\bA$, we can either look at the $\K$-automaton $\mathcal{N}_{\bU,\bA} \circledast \mathcal{A}_G$ in order to determine the exponents $\ell_{\ba,j}$ for $j\in\{1,\ldots,rs\}$, or we can use the observation that 
\[
	m(\ba)=
	\begin{pmatrix}
	\displaystyle{\sum_{\ell=1}^{rs}}M(\bz^\ell\ba)_{1,1} 
	&\cdots &
	\displaystyle{\sum_{\ell=1}^{rs}}M(\bz^\ell\ba)_{1,rs} 
	\end{pmatrix}
\]
since for all $j\in\{1,\ldots,rs\}$, $\ell_{\ba,j}\le rs$ and $M(\bz^\ell\ba)_{1,j}=0$ if $\ell\ne \ell_{\ba,j}$.

In practice, removing non-accessible and non-coaccessible states often drastically reduces the size of the linear representation of $V_{f,\bA}$.

\begin{example}
The linear representation $(\lambda',\mu',\gamma')$ corresponding to the automaton $\mathcal{A}_{V,\{0,1,2\}}$ of Figure~\ref{Fig : ProdSerieNormalisateur} with respect to the order of the states
\[
	\alpha,(1,W),(1,X),(1,Z),(2,Y),(3,X),(3,Y),(3,Z),(4,X),(5,X)
\]
is given by 
$\lambda'=\left(\begin{smallmatrix}
	1 & 1 & 0 & 0 & 0 & 0 & 0 & 0 & 0 & 0
	\end{smallmatrix}\right),\ 
\gamma'=\left(\begin{smallmatrix}
	0 & 1 & 1 & 1 & 1 & 0 & 0 & 0 & 0 & 1
	\end{smallmatrix}\right)^\intercal$
and
\[
	\mu'(0)=\left(\begin{smallmatrix}
	0 & 0 & 0 & 0 & 0 & 0 & 0 & 0 & 0 & 0\\
	0 & 1 & 0 & 0 & 0 & 3 & 1 & 0 & 0 & 0\\
	0 & 0 & 1 & 0 & 0 & 3 & 0 & 0 & 0 & 0\\
 	0 & 0 & 1 & 0 & 0 & 0 & 0 & 1 & 0 & 0\\
	0 & 0 & 1 & 0 & 0 & 0 & 3 & 0 & 0 & 0\\
	0 & 0 & 0 & 0 & 0 & 0 & 0 & 0 & 0 & 0\\
	0 & 0 & 0 & 0 & 0 & 0 & 0 & 0 & 0 & 0\\
	0 & 0 & 0 & 0 & 0 & 0 & 0 & 0 & 0 & 0\\
	0 & 0 & 0 & 0 & 0 & 0 & 0 & 0 & 0 & 0\\
	0 & 0 & 3 & 0 & 0 & 0 & 2 & 0 & 0 & 0
	\end{smallmatrix}\right),\ 	
	\mu'(1)=\left(\begin{smallmatrix}
	0 & 0 & 0 & 0 & 0 & 0 & 0 & 0 & 0 & 0\\
	0 & 0 & 3 & 1 & 0 & 2 & 0 & 0 & 0 & 0\\
	0 & 0 & 3 & 0 & 0 & 0 & 2 & 0 & 0 & 0\\
	0 & 0 & 0 & 1 & 0 & 0 & 0 & 0 & 0 & 0\\
	0 & 0 & 0 & 0 & 3 & 0 & 0 & 0 & 0 & 0\\
	0 & 0 & 0 & 0 & 0 & 0 & 0 & 0 & 0 & 0\\
	0 & 0 & 0 & 0 & 0 & 0 & 0 & 0 & 0 & 0\\
	0 & 0 & 0 & 0 & 0 & 0 & 0 & 0 & 0 & 0\\
	0 & 0 & 0 & 0 & 0 & 0 & 0 & 0 & 1 & 0\\
	0 & 0 & 0 & 0 & 2 & 0 & 0 & 0 & 0 & 0				\end{smallmatrix}\right),\	
	\mu'(2)=\left(\begin{smallmatrix}
	0 & 0 & 0 & 0 & 0 & 0 & 0 & 0 & 4 & 0\\
	0 & 0 & 0 & 0 & 2 & 0 & 0 & 0 & 0 & 0\\
	0 & 0 & 0 & 0 & 2 & 0 & 0 & 0 & 0 & 0\\
	0 & 0 & 0 & 0 & 2 & 0 & 0 & 0 & 0 & 0\\
	0 & 0 & 0 & 0 & 0 & 0 & 0 & 0 & 0 & 0\\
	0 & 0 & 0 & 0 & 0 & 0 & 0 & 0 & 1 & 0\\
	0 & 0 & 0 & 0 & 0 & 0 & 0 & 0 & 1 & 0\\
	0 & 0 & 0 & 0 & 0 & 0 & 0 & 0 & 1 & 0\\
	0 & 0 & 0 & 0 & 0 & 0 & 0 & 0 & 3 & 1\\
	0 & 0 & 0 & 0 & 0 & 0 & 0 & 0 & 0 & 0
	\end{smallmatrix}\right).	
\]
\end{example}

\section{Acknowledgment}
Célia Cisternino is supported by the FNRS Research Fellow grant 1.A.564.19F. At the time this paper was submitted, Manon Stipulanti was a postdoc at Hofstra University (New York, USA) supported by a Francqui Foundation Fellowship of the Belgian American Educational Foundation. She is currently supported by the FNRS Research grant 1.B.397.20.
\bibliographystyle{plain}
\bibliography{CharlierCisterninoStipulanti.bib}
\end{document}